\documentclass[11pt,a4paper]{amsart}
\usepackage[english,frenchb]{babel}
\usepackage[utf8]{inputenc}  
\usepackage[T1]{fontenc}  
\usepackage{amscd,amssymb,amsmath,latexsym,url,mathrsfs,graphicx,upgreek}
\usepackage[all]{xy}
\usepackage[section]{algorithm}
\usepackage{algpseudocode}

\usepackage{hyperref} \hypersetup{breaklinks=true}

\def\ov#1{{\overline{#1}}}

\def\wt#1{{\widetilde{#1}}}

\newcommand{\rank}{{\operatorname{rank}}}
\newcommand{\h}{{\operatorname{h}}}
\newcommand{\supp}{{\operatorname{supp}}}

\newcommand{\Xo}{{\X^{\rm o}}}
\newcommand{\Wo}{{\W^{\rm o}}}
\newcommand{\Qb}{\overline{\mathbb Q}}
\newcommand{\Gm}{{\mathbb G}_{\rm m}}

\newcommand{\e}{{\varepsilon}}

\newcommand{\size}{{\operatorname{size}}}

\renewcommand{\Im}{{\operatorname{im}}}

\newcommand{\hooklongrightarrow}{\lhook\joinrel\longrightarrow}

\newcommand{\SL}{{\operatorname{SL}}}

 \newcommand{\C}{{\mathbb{C}}}
\newcommand{\D}{{\mathcal{D}}} \newcommand{\E}{{\mathbb{E}}} 
\newcommand{\K}{{\mathbb{K}}} 
 \renewcommand{\P}{{\mathbb{P}}}
\newcommand{\Q}{{\mathbb{Q}}} 

\newcommand{\T}{{T}}
\newcommand{\Tp}{{T'}}

\newcommand{\Z}{{\mathbb{Z}}}

\newcommand{\Cu}{{\mathcal{U}}} 
\newcommand{\W}{{\mathcal{W}}} \newcommand{\X}{{\mathcal{X}}}
\newcommand{\Y}{{\mathcal{Y}}} 

\newcommand{\ag}{{\boldsymbol{a}}}
\newcommand{\bg}{{\boldsymbol{b}}}

\newcommand{\bfg}{{\boldsymbol{g}}}

\newcommand{\xg}{{\boldsymbol{x}}}
\newcommand{\yg}{{\boldsymbol{y}}}

\newcommand{\Bg}{{\boldsymbol{B}}}

\newcommand{\gammag}{{\boldsymbol{\gamma}}}

\renewcommand{\thetag}{{\boldsymbol{\theta}}}

\newcounter{thm}
\numberwithin{equation}{section}
\numberwithin{thm}{section}

\theoremstyle{definition}

\newtheorem{remark}[thm]{Remark}
\newtheorem{example}[thm]{Example}

\theoremstyle{plain}
\newtheorem{lemma}[thm]{Lemma}

\newtheorem{theorem}[thm]{Theorem}
\newtheorem{corollary}[thm]{Corollary}

\newtheorem{prop-def}[thm]{Proposition-Definition}

\begin{document}

\selectlanguage{english}

\title[Unlikely intersections and multiple roots]{Unlikely
  intersections and multiple roots of sparse polynomials}

\author[Amoroso]{Francesco Amoroso}
\address{Laboratoire de math\'ematiques Nicolas Oresme, CNRS 
UMR 6139, Universit\'e de Caen. BP 5186, 14032 Caen Cedex, France}
\email{francesco.amoroso@unicaen.fr}
\urladdr{\url{http://www.math.unicaen.fr/~amoroso/}}

\author[Sombra]{Mart{\'\i}n~Sombra}
\address{ICREA \&
Departament d'{\`A}lgebra i Geometria, Universitat de Barcelona.
Gran Via~585, 08007 Barcelona, Spain}
\email{sombra@ub.edu}
\urladdr{\url{http://atlas.mat.ub.es/personals/sombra/}}

\author[Zannier]{Umberto Zannier}
\address{Scuola Normale Superiore, Classe di Scienze. Piazza dei
  Cavalieri 7, 56126 Pisa, Italy}
\email{u.zannier@sns.it}
\urladdr{\url{http://www.sns.it/didattica/scienze/menunews/personale/docenti/zannierumberto/}}

\date{\today} \subjclass[2010]{Primary 11C08; Secondary 11G50.}
\keywords{Sparse polynomial, multiple roots, unlikely
  intersections.}

\thanks{This research was partially financed by the European project
  ERC Advanced Grant \og Diophantine problems\fg{} (grant agreement
  n$^{\circ}$ 267273), the CNRS project PICS 6381 \og G\'eom\'etrie
  diophantienne et calcul formel\fg{}, and the Spanish project MINECO
  MTM2012-38122-C03-02.}

\begin{abstract}
  We present a structure theorem for the multiple non-cyclotomic
  irreducible factors appearing in the family of all univariate
  polynomials with a given set of coefficients and varying exponents.
  Roughly speaking, this result shows that the multiple non-cyclotomic
  irreducible factors of a sparse polynomial, are also
  sparse.

  To prove this, we give a variant of a theorem of Bombieri and
  Zannier on the intersection of a fixed subvariety of codimension 2
  of the multiplicative group with all the torsion curves, with bounds
  having an explicit dependence on the height of the subvariety. We
  also use this latter result to give some evidence on a conjecture of
  Bolognesi and Pirola.
\end{abstract}

\maketitle

\section{Introduction}

This text is motivated by the following question: let $f\in
\Qb[t^{\pm 1}]$ be a sparse Laurent polynomial, that is, a polynomial
of high degree but relatively few nonzero terms. When does $f$ have a
multiple root in $\Qb^{\times}$?

In more precise terms, we consider sparse Laurent polynomials
given by the restriction of a \emph{fixed} regular function on
$\Gm^N$, namely 
a multivariate Laurent polynomial, to a \emph{varying} 1-parameter
subgroup.  Let $N\ge 1$ and $\gammag=(\gamma_0,\gamma_1,\dots,\gamma_{N})\in
\Qb^{N+1}$. For $\ag=(a_{1},\dots, a_{N})\in \Z^{N}$ set
$$
f_{\ag}=\gamma_0+\gamma_1 t^{a_1}+\cdots+\gamma_N t^{a_N} \in
\Qb[t^{\pm 1}]. 
$$
This  Laurent polynomial  the restriction of the affine multivariate polynomial
\begin{displaymath}
  L=\gamma_0+\gamma_1 x_{1}+\cdots+\gamma_N x_{N}\in \Qb[x_{1},\dots, x_{N}]
\end{displaymath}
to the subgroup of the multiplicative group $\Gm^{N}=(\Qb^{\times})^{N}$ parameterized by the
monomial map $t\mapsto (t^{a_{1}},\dots, t^{a_{N}})$.

The occurrence of many Laurent polynomials of the form $f_{\ag}$ with
a multiple root certainly happens in the following situation. Let $1 \le k\le N-1$ be an integer, 
$\bg_1,\ldots,\bg_N\in\Z^{N-k}$ and $\yg=(y_{1},\dots, y_{N-k})$ be a group of $N-k$ variables. Consider the Laurent polynomial
\begin{equation}\label{eq:8}
  {F}=\gamma_0+\gamma_1\yg^{\bg_1}+\cdots+\gamma_N\yg^{\bg_N} \in
\Qb[y_{1}^{\pm 1},\dots, y_{N-k}^{\pm 1}] 
\end{equation}
with  $\yg^{\bg_{i}}=y_1^{b_{i,1}}\cdots
y_{N-k}^{b_{i,N-k}}$. Suppose that $F$ has a multiple nontrivial factor $P$. 
Let  $\thetag\in \Z^{N-k} $ such that $P(t^{\theta_{1}}, \dots, t^{\theta_{N}})$ is not a monomial. Then,
for $a_i=\langle\bg_{i},\thetag\rangle$, we have 
\begin{displaymath}
 f_{\ag}=F(t^{\theta_{1}}, \dots, t^{\theta_{N}}) 
\end{displaymath}
and every root of $P(t^{\theta_{1}}, \dots, t^{\theta_{N}})$ is a
multiple root of $f_{\ag}$.

Indeed, our main result (Theorem \ref{multiple}) shows that there is a
\emph{finite} family of multivariate Laurent polynomials as in
\eqref{eq:8} such that all multiple \emph{non-cyclotomic } roots
occurring in the family of polynomials $f_{\ag}$, $\ag\in\Z^{N}$, come
by restricting the multiple factors in this finite family to a
1-parameter subgroup, as explained above. In particular, the multiple
non-cyclotomic irreducible factors of the $f_{\ag}$'s are also sparse, in
the sense that they are the restriction of a fixed Laurent polynomial
to a varying 1-parameter subgroup of $\Gm^{N}$.

The following is a precise statement of this result. 
For $\ag\in \Z^{N}$, we denote
by $\vert \ag\vert$ the maximum of the absolute values of the coordinates of
this vector. 
We also denote by $\upmu_{\infty}$ the subgroup of $\Qb^{\times}$ of roots
of unity.

\begin{theorem}
\label{multiple}
Let $N\ge 1$ and $\gammag=(\gamma_0,\gamma_1,\ldots,\gamma_N)\in
\Qb^{N+1}$. There exists an effectively computable constant~$C$ depending only on $N$ and
$\gammag$ such that the following holds.

Let $\ag=(a_1,\ldots,a_N)\in\Z^N$ such that the Laurent polynomial 
$$
f_{\ag}=\gamma_0+\gamma_1t^{a_1}+\cdots+\gamma_Nt^{a_N}\in \Qb[t^{\pm 1}]
$$ 
is nonzero and has a multiple root $\xi \in \Qb^{\times}\setminus \upmu_{\infty}$. 
Then there exist $1 \le k\le N-1$ and
$\bg_1,\ldots,\bg_N,\thetag\in\Z^{N-k}$ such that
\begin{enumerate}
\item \label{item:1} $\vert\bg_{i}\vert\leq C$, $i=1,\dots, N$, and
  $\vert\thetag\vert\le C |\ag|$;
\item \label{item:2} the matrix $\Bg=(b_{i,j})_{i,j}\in
  \Z^{N\times(N-k)} $ is \emph{primitive}, in the sense that it can be
  completed to  a matrix in ${\SL}_N(\Z)$, and $\ag=\Bg\cdot\thetag$;
\item \label{item:3} the Laurent polynomial 
$F=\gamma_0+\gamma_1\yg^{\bg_1}+\cdots+\gamma_N\yg^{\bg_N} \in\Qb[y_{1}^{\pm1},\dots, y_{N-k}^{\pm1}]$ has a multiple factor $P$ such that $\xi$
is a root of $P(t^{\theta_{1}}, \dots, t^{\theta_{N}})$.
\end{enumerate}
\end{theorem}

The situation is different for multiple cyclotomic roots. The
following example shows that the hypothesis that the root $\xi$ is not
cyclotomic is necessary for the conclusion of this result to hold.  

\begin{example}
\label{exm:1}
Let $\ag=(a_{1},a_{2},a_3)\in \Z^{3}$ coprime with $0< a_{1} <a_{2}$, $a_3=a_1+a_2$ and 
$a_3\gg0$. Consider the polynomial
\begin{displaymath}
f_{\ag}=1-t^{a_1}-t^{a_2}+t^{a_{1}+a_{2}}\in\Q[t^{\pm 1}],
\end{displaymath}
which has $\xi=1$ as a double root. 

In the notation in Theorem~\ref{multiple}, we have  $N=3$ and
$k=1,2$. The case $k=2$ is easily discarded since then, by the
conditions in \eqref{item:2}, the polynomial
$F$ coincides with~$f$, and so its degree cannot be bounded above
independently of $\ag$.

Hence we only have to consider the case $k=1$. Let
$\bg_{1},\bg_{2},\bg_{3},\thetag \in \Z^{2}$ with $\bg_{i}$ bounded
above and such that
\begin{equation}
  \label{eq:1}
  \langle \bg_{1},\thetag\rangle= a_{1},\quad   \langle \bg_{2},\thetag\rangle= a_{2},\quad 
  \langle \bg_{3},\thetag\rangle= a_{1}+a_{2}.
\end{equation}
Write $F=1-\yg^{\bg_1}-\yg^{\bg_2}+\yg^{\bg_{3}}\in
\Q[y_{1}^{\pm1},y_{2}^{\pm1}]$.

By \eqref{eq:1}, we have that $\theta_{1}$ and $\theta_{2}$ are
coprime and $\bg_{3} -\bg_{1}-\bg_{2}=\lambda
(\theta_{2},-\theta_{1})$ with $\lambda\in \Z$. Since $\bg_{i}$'s are
bounded above, $\bg_{3}=\bg_{1}+\bg_{2}$ and so
\begin{displaymath}
  F=1-\yg^{\bg_1}-\yg^{\bg_2}+\yg^{\bg_{1}+\bg_{2}}=(1-\yg^{\bg_1})(1-\yg^{\bg_2}).
\end{displaymath}
By the conditions in \eqref{item:2}, $\bg_{1}$ and $\bg_{2}$ are
linearly independent, and so $F$ has no multiple factor.
Hence, the presence of the double root $\xi=1$ cannot be explained in
this example as
coming from a multiple factor of a multivariate Laurent polynomial of
low degree restricted to a 1-parameter subgroup.
\end{example}

Theorem \ref{multiple} restricts the possible exponents $\ag\in\Z^{N}$ whose
associate polynomial has a multiple non-cyclotomic root, to
a finite union of proper linear subspaces of $\Z^{N}$. 
 
\begin{corollary}
  \label{cor:1}
Let $N\ge 1$ and $\gammag=(\gamma_0,\gamma_1,\ldots,\gamma_N)\in
\Qb^{N+1}$. Then the set of vectors 
 $\ag=(a_1,\ldots,a_N)\in \Z^{N}$ such that the Laurent polynomial 
$$
\gamma_0+\gamma_1t^{a_1}+\cdots+\gamma_Nt^{a_N}\in\Qb[t^{\pm1}]
$$ 
is nonzero and has a multiple non-cyclotomic root, 
is contained in a finite union of proper linear subspaces  of $\Z^{N}$.
\end{corollary}


To prove Theorem \ref{multiple}, we give a version of a theorem of
Bombieri and Zannier on the intersection of a subvariety of
codimension 2 of the multiplicative group with all the torsion curves,
with bounds having an explicit dependence on the height of the
subvariety (Theorem~\ref{BZ+}).  This allows us to prove a general
result concerning the greatest common divisor of two sparse
polynomials with coefficients of low height
(Theorem~\ref{gcd-ALS+}). These two theorems are presented in
\S~\ref{notation} and proved in \S~\ref{proof-BZ+} and
\S~\ref{proof-gcd-ALS+}, respectively.  Theorem~\ref{multiple} is an
easy consequence of the latter result, as shown in
\S~\ref{proof-multiple}.  Theorem~\ref{gcd-ALS+} is also used in
\S~\ref{proof-pirola} to prove Theorem~\ref{pirola},  giving some
evidence on a conjecture of Bolognesi and Pirola
\cite{BolognesiPirola:osde}.

\medskip \noindent {\bf Acknowledgments.}  Part of this work was done
while the authors met at the Scuola Normale Superiore (Pisa), the
Universitat de Barcelona, and the Université de Caen. We thank these
institutions for their hospitality.

\section{Intersections of subvarieties with torsion curves and gcd of sparse
  polynomials of low height}
\label{notation}

We first recall some definitions and basic facts.  Boldface letters
denote finite sets or sequences of objects, whose the type and number
should be clear from the context: for instance, $\xg$ might denote the
group of variables $(x_1,\dots,x_n)$, so that $\Qb[\xg^{\pm1}]$
denotes the ring of Laurent polynomials
$\Qb[x_1^{\pm1},\dots,x_n^{\pm1}]$. Given a vector
$\ag=(a_1,\ldots,a_N)\in\Z^N$ we set
\begin{displaymath}
  \vert\ag\vert=\max_{j}\vert
a_j\vert.
\end{displaymath}
Given a group
homomorphism $\varphi\colon \Gm^{n}\to \Gm^{N}$,  there exist unique vectors
$\bg_1,\ldots,\bg_N\in\Z^n$ such that
$\varphi(\xg)=(\xg^{\bg_1},\ldots,\xg^{\bg_N})$ for all $\xg\in
\Gm^{n}$. We define the \emph{size} of $\varphi$ as
\begin{displaymath}
  \size(\varphi)= \max_j|\bg_{j}|
\end{displaymath}
We also denote by
\begin{displaymath}
 \varphi^{\#}\colon \Qb[y_{1}^{\pm1},\dots,
 y_{N}^{\pm1}]\longrightarrow \Qb[x_{1}^{\pm1},\dots, x_{n}^{\pm1}]  ,
 \quad y_{i}\longmapsto \xg^{\bg_{i}}
\end{displaymath}
the associated morphism of algebras. If $\psi\colon \Gm^{N}\to
\Gm^{M}$ is a further homomorphism, then $(\psi\circ\varphi)^{\#}=
\varphi^{\#}\circ \psi^{\#}$.

Let $D\ge 1$ and $f_1$, $f_2\in \Z[t]$ polynomials of degree $\leq D$
with fixed coefficients and fixed number of nonzero terms. Filaseta,
Granville and Schinzel have shown that, if either $f_1$ or $f_2$ do
not vanish at any root of unity, then the greatest common divisor
$\gcd(f_1,f_2)$ can be computed in time polynomial in $\log (D)$
\cite{FilasetaGranvilleSchinzel:igcdasp}.  More recently, Amoroso,
Leroux and Sombra gave an improved version of this result
\cite{ALS:oslpe}. The following is its precise statement.

\begin{theorem}[\cite{ALS:oslpe}, Theorem 4.3]
\label{gcd-ALS} There is an algorithm that, given a number field $\K $
and polynomials $f_1$, $f_2\in \K [t]$, computes a polynomial $p\in \K
[t]$ dividing $\gcd(f_1,f_2)$ and such that $\gcd(f_1,f_2)/p$ is a product of cyclotomic
polynomials. 

If both $f_1$ and $f_2$ have  degree
bounded by $D$, height bounded by $h_0$ and number of nonzero coefficients
bounded by $N$, this computation is done with $O_{\K ,N,h_0}(\log
(D))$ bit operations. 
\end{theorem}


In more detail, write 
\begin{equation*}
f_i=\gamma_{i,0}+\gamma_{i,1}t^{a_1}+\cdots+\gamma_{i,N}t^{a_N}\in
\K[t], \quad i=1,2, 
\end{equation*}
with $a_{j}\in \Z$ and $\gamma_{i,j}\in \K$. 
Denote by $\varphi\colon\Gm\to\Gm^N$ the homomorphism given by
$\varphi(t)=(t^{a_1},\ldots,t^{a_N})$ and  set
\begin{equation*}
L_i=\gamma_{i,0}+\gamma_{i,1}x_1+\cdots+\gamma_{i,N}x_N, \quad i=1,2, 
\end{equation*}
so that $f_i=\varphi^\#(F_i)$.  Then, the algorithm underlying
Theorem~\ref{gcd-ALS} computes an integer $0 \le k\le N-1$ and two
homomorphisms $\psi\colon\Gm^{N-k}\to\Gm^N$ and
$\varphi_1\colon\Gm\to\Gm^{N-k}$ with $\psi$ injective, such that
$\psi\circ\varphi_1=\varphi$ and
\begin{displaymath}
p=\varphi_1^\#(\gcd(\psi^\#(L_1),\psi^\#(L_2))).  
\end{displaymath}
Moreover, the size
of $\psi$ and $\varphi_{1}$ is respectively bounded by $B$ and $BD$,
where $B$ is a constant depending only on $\K$, $N$ and $h_0$.

This algorithm relies heavily on a former conjecture of Schinzel on
the intersection of a subvariety of the multiplicative group with
1-parameter subgroups. This conjecture was proved by Bombieri and
Zannier in \cite[Appendix]{Schinzel:psrr}.  For the
reader's convenience, we recall an improved version of this result.

\begin{theorem}[\cite{BombieriMasserZannier:assta}, Theorem 4.1]
\label{BZ}
Let $N\ge 1$ and $P, Q\in \Qb[x_{1},\dots, x_{N}] $ coprime
polynomials. Then there exists an effectively computable constant $B$
depending only on $P$ and $Q$ with the following property.

Let $a_j\in \Z$, $j=1,\dots, N$, $\zeta_j\in \upmu_{\infty}$ and
$\xi\in \C^{\times}$ with
$$
P(\zeta_1\xi^{a_1},...,\zeta_N\xi^{a_N}) = Q(\zeta_1\xi^{a_1},...,\zeta_N\xi^{a_N}) = 0.
$$ 
Then there exist  $b_j\in \Z$, $j=1,\dots, N$,  with $0 < \max_{j}|b_j| \leq B$ and
$$
\prod_{j=1}^{N}(\zeta_j\xi^{a_j})^{b_j}= 1.
$$ 
In particular, if $\xi\notin \upmu_{\infty}$, then $\sum_{j=1}^{N}a_jb_j=0$.
\end{theorem}

We are interested in extension of Theorem~\ref{gcd-ALS} to polynomials
$f_1$, $f_2$ having low, but unbounded, height. To this end, we need
first a version of Theorem~\ref{BZ} with explicit dependence on the
height of the input polynomials $P$ and $Q$.

As already remarked by Schinzel, the constant $B$ in this theorem
cannot depend only on $N$, on the field of definition and on the
degrees of $P$ and $Q$. For instance, for the data
\begin{equation*}
N=2, \quad P(x,y)=x-2 , \quad 
Q(x,y)=y-2^a \quad \text{ and } \quad (\zeta_1\xi^{a_1},\zeta_2\xi^{a_2})=(2,2^a)  ,
\end{equation*}
one has $B(P,Q)\geq a$. I delete the reference to~\cite[page
 7]{BombieriMasserZannier:assta} because it does not give any further
 details with respect to the example. Maybe we should add the
 reference to the paper of Schinzel with this example, but I don't
 have this reference.

The following result gives, under some restrictive hypothesis, the
dependence of the constant $B$ on the height of the input
polynomials. Recall that a \emph{coset} of $\Gm^N$ is a translate of a
subtorus, and that a \emph{torsion coset} is a translate of a subtorus
by a torsion point. A \emph{torsion curve} (respectively, a
\emph{torsion hypersurface}) is a torsion coset of dimension 1
(respectively, of codimension
1). Following~\cite{BombieriZannier:apsG}), given a subvariety~$\X$ of
$\Gm^N$, we denote by $\Xo$ the complement in $\X$ of the union of all
cosets of positive dimension contained in $\X$.

We consider the standard compactification of the multiplicative group
given by the inclusion
\begin{displaymath}
 \iota\colon
\Gm^N\hooklongrightarrow\P^N \quad, \quad (x_{1},\dots, x_{N})\longmapsto
(1:x_{1}:\cdots:x_{N}).  
\end{displaymath}
We define the \emph{degree} of an irreducible subvariety
$\X$ of $\Gm^{N}$, denoted by $\deg(\X)$, as the degree of the Zariski
closure $\ov{\iota(\X)}\subset \P^N$, and the \emph{height} of a point
$\xi\in \Gm^{N}$, denoted by $\h(\xi)$, as the Weil height of the
projective point $\iota(\xi)\in \P^{N}$.

\begin{theorem}
\label{BZ+}
Let $\X\subset\Gm^N$ be a subvariety defined over a number field of degree $\delta$ by polynomials of
degree bounded by $ d_0$ and
height bounded by $ h_0$. Let $0< \varepsilon<1$. Then
there exists {an effectively computable} constant $B$  depending only on $N$,
$d_0$, $\delta$ and $\varepsilon$, with the following property. 

Let $\W$ be an irreducible component of $\X$ of codimension at least
$2$, $\T$ a torsion curve and $\xg\in\Wo\cap \T$ a non-torsion point.
Then either
$$
\deg(\T)^{\frac{1-\varepsilon}{N-1}}\leq B\cdot (1+h_0)
$$
or there exists   a torsion hypersurface $\Tp$ with 
 $\xg\in \Tp$ and $\deg(\Tp)\le B$.
\end{theorem}

\begin{remark} 
\label{comment}
  We might restate Theorem~\ref{BZ+} in a slightly different way in the
  case when the torsion curve $\T$ is a subtorus. Let
  $\varphi\colon\Gm\to\Gm^N$ be an injective homomorphism and keep $\X$, $\W$ and
  $\varepsilon$ as in the statement of the theorem.  Let
  $\xi\in\Qb^{\times}\setminus \upmu_{\infty}$ such that
  $\varphi(\xi)\in\Wo$. In this situation, Theorem~\ref{BZ+}
can be reformulated to the statement that, if
$$
\size(\varphi) ^{\frac{1-\varepsilon}{N-1}}> B\cdot (1+h_0),
$$ 
then $\xg$ is contained in a subtorus $\Tp$ of codimension $1$ and
degree bounded by~$ B$.


Indeed, 
Theorem \ref{BZ+} applied to  the subtorus  $\T=\Im(\varphi)$, shows that $\varphi(\xi)\in \Tp$ for a torsion
hypersurface $\Tp$ of degree bounded by $B$. This torsion hypersurface is
defined by the single equation $\xg^\bg=\omega$ for some $\bg\in\Z^N$
with $\vert\bg\vert\leq B$ and $\omega\in \upmu_{\infty}$. Write 
$\varphi(t)=(t^{a_1},\ldots,t^{a_N})$ with $a_{i}\in\Z$. Then
$$
\xi^{a_1b_1+\cdots+a_Nb_N}=\omega.
$$
Since $\xi$ is not torsion,  $\sum_{j}a_jb_j=0$ and
$\omega=1$. Hence, $\Tp$ is a subtorus and
$\Im(\varphi)\subseteq \Tp$.
\end{remark}

The following variant of Schinzel's example shows that the hypothesis
that $\xg\in\Xo$ is necessary for the conclusion of Theorem \ref{BZ+}
to hold.

\begin{example}
\label{exm:2}
Let $1\le a \le b$ and consider the irreducible subvariety
\begin{displaymath}
\X=\{(2,2^a)\}\times\Gm\subset \Gm^{3}.    
\end{displaymath}
With notation as in Theorem \ref{BZ+}, we have $N=3$, $d_0=1$ and
$h_0\approx a$. Since $\X$ is a coset of positive dimension,
$\Xo=\emptyset$. Let $\T\subset \Gm^{3}$ be the subtorus
parameterized by $t\mapsto (t,t^a,t^b)$ and pick the point
$\xg=(2,2^a,2^b)\in \X\cap T$. It is easy to verify that, for any
fixed $0< \varepsilon<1$ and $B>0$, if $a$ and $b/a$ are
sufficiently large, then neither $\deg (\T)^{\frac{1-\e}{2}}\leq B\cdot
(1+h_0)$ nor $\xg\in \Tp$ for any torsion hypersurface of degree
bounded by $ B$.
\end{example}

Theorem~\ref{BZ+} allows us to prove the desired extension of
Theorem~\ref{gcd-ALS} to polynomials of low height. The following statement
gives the quantitative aspects of this result. 

\begin{theorem}
\label{gcd-ALS+} 
Let $\K $ be a
number field of degree $\delta$. For a family of elements $\gamma_{i,j}\in \K$, $i=1,\dots, s$, 
$j=1,\dots,N$, and a sequence of $N\ge 1$ coprime integers $a_1,\dots,a_N$, 
we consider the system of Laurent polynomials
$$
f_i=\gamma_{i,0}+\gamma_{i,1}t^{a_1}+\cdots+\gamma_{i,N}t^{a_N}, \qquad i=1,\ldots,s.
$$
We assume $f_1,\ldots,f_s$ not all zeros. Set
$$
L_i=\gamma_{i,0}+\gamma_{i,1}x_1+\cdots+\gamma_{i,N}x_N, \qquad i=1,\ldots,s,
$$ 
and let $\varphi\colon\Gm\to\Gm^N$ be the homomorphism given by
$\varphi(t)=(t^{a_1},\ldots,t^{a_N})$. Put $D=\vert\ag\vert$ and
$h_{0}=\max_{i,j} \h(\gamma_{i,j})$.

Then there exists {an effectively computable} constant $B'$ depending only on $N$ and $\delta$,
with the following property. If
\begin{equation}
\label{main-inequa}
D^{\frac{1}{2(N-1)}}>B'\cdot (1+h_0),
\end{equation}
then there exist $0 \le k\le N-1$ and homomorphisms
\begin{displaymath}
\psi\colon\Gm^{N-k}\to\Gm^N \quad \text{ and } \quad \varphi_1\colon\Gm\to\Gm^{N-k}  
\end{displaymath}
such that
\begin{enumerate}
\item \label{item:7} $\psi$ is injective and $\psi\circ\varphi_1=\varphi$;
\item \label{item:8} $\size(\psi)\le B'$ and $\size(\varphi_1) \le  B'D$;
\item \label{item:9} Set 
$$
G=\gcd(\psi^\#(L_1),\ldots\psi^\#(L_s)) \quad \text{ and } \quad g=\varphi_1^\#(G).
$$ 
Then $g\mid\gcd(f_1,\ldots,f_s)$. 
Moreover, if $\xi$ is a root of
$\gcd(f_1,\ldots,f_s)/g$, then either $\xi\in \upmu_\infty$ or there
exists a nonempty proper subset $\Lambda\subset\{1,\ldots,N\}$ such
that $\gamma_{i,0}+\sum_{j\in\Lambda}\gamma_{i,j}\xi^{a_j}=0$,
$i=1,\ldots,s$.
\end{enumerate}
\end{theorem}

Similarly as for Theorem \ref{gcd-ALS}, the datum $k$, $\psi$ and
$\varphi_{1}$ can be effectively computed. In the present situation,
this is done by the procedure described in \S~\ref{proof-gcd-ALS+},
and this computation costs $O_{\delta,N,s}(\log (D))$ bit
operations.



\section{Proof of Theorem~\ref{BZ+}}
\label{proof-BZ+}

All irreducible components of $\X$ are defined over a number field of
degree bounded by $C$ by polynomials of degree bounded by $C$ and
height bounded by $C h_{0}$, for a constant $C$ depending only on $N$,
$d_{0}$ and $\delta$. Using this, we reduce without loss of generality
to the case when $\X$ is an irreducible subvariety of codimension at
least 2.

We follow closely the proof of \cite[Theorem
4.1]{BombieriMasserZannier:assta}.
Since we assume that
$\xg\in\Xo\cap \T$, the first reduction of the proof in
\emph{loc. cit.} is 
unnecessary in our present situation. Write 
$$
\T=\{(\zeta_1t^{a_1},\ldots,\zeta_N t^{a_N}) \mid  t\in\Gm\}\subseteq\Gm^N
$$ 
with $a_1,\ldots,a_N\in\Z$ coprime and $\zeta_1,\ldots,\zeta_N\in
\upmu_{\infty}$. Thus $\deg (\T)=\vert\ag\vert$. As in
\emph{loc. cit.} we construct, using  geometry of numbers, a
$2$-dimensional torsion coset $T_2$ containing $\T$ and such that
\begin{equation}
\label{siegel}
\deg (T_2)\leq B_1\vert\ag\vert^{\frac{N-2}{N-1}}
\end{equation}
for a constant $B_1$ depending only on $N$.
The proof goes on by distinguishing two cases. 

Suppose first that the point  $\xg$ is an isolated component of
$\X\cap T_2$. 
Since $\xg\in \X\cap \T$, we can write
$\xg=(\zeta_1\xi^{a_1},\ldots,\zeta_N\xi^{a_N})$ with $\Qb^\times\setminus
\upmu_{\infty}$. 
Let $\K $ be a field of definition of $\X$ and set $\E=\K
(\zeta_1,\ldots,\zeta_N)$, which is a field of definition for both
$\X$ and $\T$. Put $\D=[\E(\xg):\E]$. 
Using  B\'ezout theorem and~(\ref{siegel}), we deduce that this degree
satisfies the bound
\begin{equation}
\label{bezout}
\D\leq \deg(\X\cap T_2)\leq B_1\vert\ag\vert^{\frac{N-2}{N-1}}\deg(\X).
\end{equation}
Moreover, since $a_1,\ldots,a_N$ are coprime, $[\E(\xi):\E]=\D$. 

Let $0<\varepsilon<1$. We have that $\E(\xi)$ is an extension of
degree $\leq [\K :\Q] \D$ of the cyclotomic extension
$\Q(\zeta_1,\ldots,\zeta_N)$.  By the relative Dobrowolski lower bound
of~\cite{AmorosoZannier:rDlbae}, the height of $\xi$ is bounded from
below by
\begin{equation}
\label{dobro}
\h(\xi)\geq B_2\D^{-1-\varepsilon},
\end{equation}
where $B_2$ is {an effective} constant that depends only on $\varepsilon$ and $[\K:\Q]$.

By \cite[Appendix, Theorem 1]{Schinzel:psrr}, since the point $\xg$
lies in $\Xo\cap
T$, its height is bounded above by a constant depending
only on $\X$. Indeed, a close inspection of the
proof of this  result shows that
\begin{equation}
  \label{eq:2}
\h(\xg)\leq B_3\cdot(1+h_0).   
\end{equation}
for {an effectively computable} $B_3$ that depends only on $\delta$ and $N$.
Alternatively, this can be obtained by applying Habegger's effective
version of the bounded height theorem \cite[Theorem
11]{Habegger:ehubat} with the choice of parameters $r=2$ and $s=n-1$
with respect to the notation therein,  together
with the arithmetic B\'ezout theorem in \cite[Corollary
2.11]{KPS:sean}. Thus
\begin{equation}
\label{BHT}
\vert\ag\vert \, \h(\xi)\leq \sum_{i=1}^{N}\h(\zeta_i\xi^{a_i})\leq N\h(\xg).
\end{equation}
Combining \eqref{bezout}, \eqref{dobro}, \eqref{eq:2} and~\eqref{BHT}, we get
$$
\deg(\T)= \vert\ag\vert\leq B_2^{-1}\Big(B_1\vert\ag\vert^{\frac{N-2}{N-1}}\deg(\X)\Big)^{1+\varepsilon} NB_3\cdot(1+h_0).
$$
From here, we deduce that
$$
\deg(\T)^{\frac{1-\varepsilon'}{N-1}}\leq B\cdot (1+h_0).
$$
with $\varepsilon'=(N-2)\varepsilon$ and where $B$ is any constant
$\geq B_4=B_2^{-1}(B_1\deg(\X))^{1+\varepsilon}NB_3$, which shows the
result in this case. 

Now suppose that $\xg$ lies in an irreducible component of positive
dimension of $\X\cap T_2$. Denote by $\Y$ this irreducible component,
which is thus a $\X$-anomalous subvariety. Let $\Y_{\max}$ be a a
maximal $\X$-anomalous subvariety containing $\Y$. From the
Bombieri-Masser-Zannier uniform structure theorem \cite[Theorem
1.4]{BombieriMasserZannier:assta}, this subvariety $\Y_{\max}$ is
contained in a coset $\bfg H$ whose degree is bounded in terms of $\X$.
Indeed, by the inequality~(3.4) in \cite{BombieriMasserZannier:assta}, this degree is
bounded by a constant $B_5$
depending only on $\delta$ and $\deg(\X)$. As explained in
  \emph{loc. cit.}, this constant is also effectively computable.

The intersection $T_2\cap \bfg H$ is a union of cosets associated to
the same subtorus. Denote by $K$ the unique coset in this intersection
that contains $\Y$. Its dimension is either $1$ or $2$. The case
$\dim(K)=1$ is not possible since, otherwise, $\Y=K$ is a coset, which
is forbidden by the hypothesis that $\xg\in\Xo$. Hence $\dim (K)=2$,
which means that some irreducible component of $T_2$ lies in $\bfg
H$. Take a torsion point $\bfg_0$ lying in this irreducible
component. Then $\bfg_0\in \bfg H$ and $\bfg H=\bfg_0H$ is a torsion
coset of degree bounded by~$B_5$. We can find a further constant $B_6$
depending only on $\delta$ and $\deg(\X)$ such that there exists a
torsion hypersurface $T'$ with $\bfg_0H\subseteq T'$ and $\deg(T')\leq
B_6$. We then choose $B=\max(B_4,B_6)$, concluding the proof.

\section{Proof of Theorem~\ref{gcd-ALS+}}
\label{proof-gcd-ALS+}

We follow the proof of~\cite[Theorem 4.3]{ALS:oslpe}, replacing the
use of Theorem~\ref{BZ} by Theorem~\ref{BZ+}.  We first need to prove
some auxiliary lemmas.

\begin{lemma}
\label{riduci}
Let $\varphi\colon\Gm\to\Gm^N$ be a homomorphism of size $D$ and
$T\subseteq \Gm^N$ a subtorus of codimension $1$. We can test if
$\Im(\varphi)\subseteq T$ and, if this is the case, we can compute two
homomorphisms $\wt{\psi}\colon\Gm^{N-1}\to\Gm^N$ and
$\wt{\varphi}\colon\Gm\to\Gm^{N-1}$ such that
\begin{enumerate}
\item \label{item:10} $\wt{\psi}$ is injective and $\wt{\psi}\circ\wt{\varphi}=\varphi$;
\item \label{item:11} $\size(\wt{\psi})= O(1) $ and
  $\size(\wt{\varphi})= O(D)$.
\end{enumerate}
This computation can be done with $O(\log (D))$ bit operations. All
the implicit constants depend only on $N$ and $\deg(T)$.
\end{lemma}

\begin{proof}
  Let $\xg^\bg=1$ be an equation for $T$ and write
  $\varphi(\xg)=(\xg^{a_1},\ldots,\xg^{a_N})$ with $a_{1},\dots,
  a_{N}\in \Z$ coprime. Then $\Im(\varphi)\subseteq T$ if and only if
  $\sum_{j}a_jb_j=0$. Let us assume that this is the case. We choose
  an automorphism $\tau$ of $\Gm^N$ such that $\tau(T)$ is defined by
  the equation $x_N=1$. Let $\iota\colon \Gm^{N-1}\to\Gm^{N}$ be the
  standard inclusion identifying $\Gm^{N-1}$ with the hyperplane of
  equation $x_N=1$, and consider  the
  projection onto the
  first $N-1$ coordinates 
  \begin{displaymath}
 \pi\colon \Gm^{N} \to\Gm^{N-1} \quad , \quad 
     \pi(x_1,\ldots,x_N)=(x_1,\ldots,x_{N-1},1).
  \end{displaymath}
We then set $\wt{\psi}=\tau^{-1}\circ\iota$ and
  $\wt{\varphi}=\pi\circ \tau \circ \varphi$.

  We leave to the reader the verification on the correctness and the
  complexity of this algorithm, see~\cite[Lemma 4.1]{ALS:oslpe} for
  further details.
\end{proof}





We now describe the algorithm underlying Theorem~\ref{gcd-ALS+}.

 \begin{algorithm}  \caption{} \label{algo}
    \begin{algorithmic}[1]
\Require~
a subvariety $\X\subset\Gm^N$ defined over a number field $\K $ and a homomorphism $\varphi\colon\Gm\to\Gm^N$.
\Ensure an integer $k$ with $0\le k\le N-1$ and two homomorphisms $\psi\colon\Gm^{N-k}\to\Gm^N$ and $\varphi_1\colon\Gm\to\Gm^{N-k}$ .
\State \label{ini} Set $k\leftarrow 0$, $\psi\leftarrow{\rm Id}_{\Gm^N}$ and $\varphi_1\leftarrow\varphi$;
\While{$k< N$} \label{alg:39}
\State \label{alg:43} let $B$ the constant in Theorem~\ref{BZ+} for the subvariety ${\psi^{-1}(\X)}\subset\Gm^{N-k}$\par 
\hskip-0.2cm  and the choice $\varepsilon =\frac{1}{2}$;
\State  \label{alg:44} set $\Phi\leftarrow \{\{\xg^{\bg}=1\}\mid \bg\in \Z^{N} \text{ primitive such that } |\bg|\le B\}$;
\While{$\Phi\ne \emptyset$} 
\State \label{toro} choose $\Tp\in \Phi$;
\If{$\Im(\varphi_1)\subseteq \Tp$}
\State \label{comp-morph} compute as in Lemma~\ref{riduci} homomorphisms $\wt{\psi}\colon\Gm^{N-k-1}\to\Gm^{N-k}$\par
\hskip1cm and $\wt{\varphi}\colon\Gm\to\Gm^{N-k-1}$ such that $\varphi_1=\wt{\psi}\circ\wt{\varphi}$; 
\State \label{set} set $\psi\leftarrow \psi\circ\wt{\psi}$,
$\varphi_1\leftarrow \wt{\varphi}$, $k\leftarrow k+1$, $\Phi\leftarrow
\emptyset$;
\Else 
\State set $\Phi\leftarrow \Phi\setminus \{\Tp\}$; 
\EndIf
\EndWhile
\EndWhile \label{alg:40}
    \end{algorithmic}
  \end{algorithm}

\begin{lemma}
\label{lemma1}
Let $\X\subset\Gm^N$ be a subvariety defined over a number field of
degree $\delta$ by polynomials of degree bounded by $d_0$. Let also
$\varphi\colon\Gm\to\Gm^N$ be a homomorphism of size
$D$. Algorithm~\ref{algo} computes an integer $k$ with $0\leq k<N-1$
and two homomorphisms $\psi\colon\Gm^{N-k}\to\Gm^N$ and
$\varphi_1\colon\Gm\to\Gm^{N-k}$ such that
\begin{enumerate}
\item \label{item:12} $\psi$ is injective and $\psi\circ\varphi_1=\varphi$;
\item \label{item:13} $\size(\psi)= O(1)$ and $\size(\varphi_1)= O(D)$.
\end{enumerate}
This computation is done with $O(\log D)$ bit
operations. All the implicit constants in the $O$-notation depend
only on $N$, $d_0$ and $\delta$.
\end{lemma}

\begin{proof}
  We show by induction on $k$ that the homomorphisms $\psi$ and
  $\varphi_1$ constructed by the algorithm at the level $k$ satisfy
  both \eqref{item:12} and \eqref{item:13}.

This is certainly true at the level $k=0$. Indeed at this level
$\psi={\rm Id}_{\Gm^N}$ and $\varphi_1=\varphi$.

Let $k$ be an integer with $1\leq k<N$ and assume that at the level
$k-1$ the homomorphisms $\psi$ and $\varphi_1$ satisfy \eqref{item:12}
and \eqref{item:13}. By Lemma~\ref{riduci}, the homomorphisms
$\wt{\psi}$ and $\wt{\varphi}$ at line~\ref{comp-morph} satisfy
$\wt{\psi}\circ\wt{\varphi}=\varphi_1$. Hence the updated values of
$\psi$ and $\varphi_1$, that is $\psi\circ\wt{\psi}$ and
$\wt{\varphi}$, satisfy
$$
(\psi\circ\wt{\psi})\circ\wt{\varphi}=\psi\circ\varphi_1=\varphi.
$$

Moreover, since $\psi$ and $\wt{\psi}$ are injective, by induction
and by Lemma~\ref{riduci}\eqref{item:10}, $\psi\circ\wt{\psi}$ is also
injective.

Let $B$ be as in line \ref{alg:43} of the algorithm~\ref{algo}, that is, the constant 
in Theorem~\ref{BZ+} for the subvariety $\psi^{-1}(\X)$ and the choice $\varepsilon =\frac{1}{2}$.
Since $\size(\psi)=O(1)$ and $\X$ is linear, $\psi^{-1}(\X)$ is defined over a number field
of degree $O(1)$ by polynomials of degree $O(1)$ and height
$O(h_0)$, with implicit constants depending only on $N$ and
$\delta$. In particular, $B=O(1)$. The same is therefore true for the degree of the subtorus $\Tp$ at
line~\ref{toro}. By Lemma~\ref{riduci}\eqref{item:11}, the
homomorphisms $\wt{\psi}$ and $\wt{\varphi}$ at line~\ref{comp-morph}
have size $O(1)$ and $O(D)$ respectively. Thus $\psi\circ\wt{\psi}$
and $\wt{\varphi}$ have also size $O(1)$ and $O(D)$, respectively.

We left to the reader the verification on the complexity of the algorithm.
\end{proof}

We are now able to conclude the proof of Theorem~\ref{gcd-ALS+}.
Let $\K $ and $f_1$, $\dots$, $f_s$ be as in that theorem. Thus $\K $ is a number field of degree $\delta$ and 
$$
f_i=\gamma_{i,0}+\gamma_{i,1}t^{a_1}+\cdots+\gamma_{i,N}t^{a_N},\qquad i=1,\ldots,s,
$$
are Laurent polynomials, not all zeros, with $a_1,\ldots,a_N$ coprime. Set $D=\vert\ag\vert$ 
and assume $\max_{i,j}h(\gamma_{i,j})\leq h_0$.  We consider the homomorphism
$\varphi\colon\Gm\to\Gm^N$ given by
$\varphi(t)=(t^{a_1},\ldots,t^{a_N})$. Since $a_1,\ldots,a_N$ are
coprime, $\deg(\Im(\varphi))=D$. We let
$$
L_i=\gamma_{i,0}+\gamma_{i,1}x_1+\cdots+\gamma_{i,N}x_n,\qquad i=1,\ldots,s.
$$ 
Thus $f_i=\varphi^\#(L_i)$. We apply Algorithm~\ref{algo} to the
linear subvariety $\X$ defined in $\Gm^N$ by the system of equations
$L_1=\ldots=L_s=0$.

From now on, we denote by $k\in\{0,\ldots,N-1\}$,
$\psi\colon\Gm^{N-k}\to\Gm^N$ and $\varphi_1\colon\Gm\to\Gm^{N-k}$ the
output of Algorithm \ref{algo} applied to this subvariety. Put for
short $F_i=\psi^\#(L_i)$. By Lemma~\ref{lemma1},
$\varphi_1^\#(F_i)=f_i$. Since $f_1,\ldots,f_s$ are not all zeros, the same holds for 
$F_1,\ldots,F_s$. Now set 
\begin{displaymath}
  G=\gcd(F_1,\dots, F_s) \quad \text{ and } \quad
g=\varphi_1^\#(G).
\end{displaymath}
Then $g\vert\gcd(f_1,\ldots,f_s)$, as in Theorem \ref{gcd-ALS+}\eqref{item:9}.

Let $B'$ be a constant depending only on $N$ and $\delta$ such that
\begin{equation}
\label{H}
D^{\frac{1}{2(N-1)}}>B'\cdot (1+h_0),
\end{equation}
as in the statement of Theorem \ref{gcd-ALS+}, to be fixed later on.

Let $\Omega$ be the set of points $\xi\in \C^{\times}$ which are
either a root of unity or a common root of the system of polynomials
$\gamma_{i,0}+\sum_{j\in\Lambda}\gamma_{i,j}t^{a_j}$, $i=1,\ldots,s$,
for a nonempty proper subset
$\Lambda\subset\{1,\ldots,N\}$. 

Let $\xi\not\in\Omega$ be a common zero of $f_1,\ldots,f_s$ and $\W$ a
component of $\psi^{-1}(\X)$ such that $\varphi_1(\xi)\in \W$.

We first remark that $\varphi_1(\xi)\in \W^{\rm o}$. If it is not, the
point $\yg=\varphi_1(\xi)$ is in a coset $gH\subseteq
\W\subseteq\psi^{-1}(\X)$ of positive dimension. By
Lemma~\ref{lemma1}\eqref{item:13}, the point
$\xg=\varphi(\xi)=\psi(\yg)$ is contained in the coset
$\psi(gH)\subseteq\X$, which is also of positive dimension since
$\psi$ is injective. 

The cosets included in a linear variety $\X$ have been explicitly
classified in~\cite[page 161]{Schmidt:hps}. By this result, there
exists a nonempty proper subset $\Lambda\subset\{1,\ldots,N\}$ such
that $\gamma_{i,0}+\sum_{j\in\Lambda}\gamma_{i,j}x_j=0$,
$i=1,\ldots,s$. Hence $\xi$ is a common root of
$\gamma_{i,0}+\sum_{j\in\Lambda}\gamma_{i,j}t^{a_j}$, $i=1,\ldots,s$,
but this is not possible because $\xi\notin \Omega$.

Thus $\xi$ is not a root of unity and $\varphi_1(\xi)\in \W^{\rm
o}$. We apply Theorem~\ref{BZ+} in the simplified form of
Remark~\ref{comment}, choosing $N\leftarrow N-k$, $\X\leftarrow \psi^{-1}(\X)$,
$\varepsilon\leftarrow 1/2$ and $\varphi\leftarrow\varphi_1$.
Let $B$ be as in line \ref{alg:43} of the algorithm~\ref{algo}. As already remarked in 
the proof of Lemma~\ref{lemma1}, $\psi^{-1}(\X)$ is defined over a number field
of degree $O(1)$ by polynomials of degree $O(1)$ and height
$O(h_0)$, with implicit constants depending only on $N$ and
$\delta$. In particular, $B=O(1)$. 
By the quoted Remark~\ref{comment}, one of the following assertions holds:
\begin{enumerate}
\item \label{item:4} there exists a subtorus $\Tp$ of codimension $1$
and degree bounded by $B$ such that $\Im(\varphi_1)\subseteq \Tp$;
\item \label{item:5} $\deg(\Im(\varphi_1))^{\frac{1}{2(N-k-1)}}=O(1+h_0)$;
\item \label{item:6} $\W$ has codimension $1$.
\end{enumerate}
By construction, \eqref{item:4} is not possible because $\Tp\in
\Phi$. Let us assume that  \eqref{item:5} holds. By Lemma~\ref{lemma1},
$D=\deg(\Im(\varphi))=\deg(\Im(\psi\circ\varphi_1))=O(\deg(\Im(\varphi_1)))$. Thus
$$
D^{\frac{1}{2(N-1)}}\leq D^{\frac{1}{2(N-k-1)}}=O(\deg(\Im(\varphi_1))^{\frac{1}{2(N-k-1)}})
=O(1+h_0).
$$
Choosing the constant $B'$ sufficiently large, this
contradicts the inequality~\eqref{H}. Thus \eqref{item:6} must hold and $\W$
has codimension $1$.

This discussion implies that the ideal
$(F_1,\ldots,F_s)\subset\K[y_{1}^{\pm1},\dots, y_{N-k}^{\pm1}]$
becomes principal when restricted to a suitable neighborhood
$U\subset\Gm^{N-k}$ of
$\psi^{-1}(\X)\setminus\varphi_1(\Omega)$. Hence,
$(F_1,\ldots,F_s)=(G)$ for some Laurent polynomial $G$ on that
neighborhood. We deduce that $\varphi_1^{-1}(U)$ is a neighborhood of
the set of common zeros $\xi\not\in\Omega$ of $f_1,\ldots,f_s$ and
$(f_1,\ldots,f_s)= (g)$ on $\varphi_1^{-1}(U)$. This completes the
proof of the theorem.

\begin{remark} \label{rem:1} For the study of multiple roots of sparse
  polynomials and, in particular, to prove Theorem \ref{gcd-ALS+}, it
  is not enough to dispose of a version of Theorem~\ref{BZ} with an
  explicit dependence of its constant $B$ on the height of the input
  polynomials. We really need the dichotomy that appears in
  Theorem~\ref{BZ+}, with a bound for the degree of $\Tp$ independent
  of the height of the equations defining $\X$, whenever the degree of
  the torsion curve $T$ is large enough.

  In any case, it is possible to adapt the proof of \cite[Theorem
  4.1]{BombieriMasserZannier:assta} to prove
  such an effective version of Theorem~\ref{BZ}.
\end{remark}

\section{Proof of Theorem~\ref{multiple}}
\label{proof-multiple}

Let $N\ge 1$ and $\gammag=(\gamma_0,\gamma_1,\ldots,\gamma_N)\in
\Qb^{N+1}$. Consider the number field $\K=\Q(\gammag)$ and
the affine polynomial
\begin{displaymath}
  L=\gamma_0+\gamma_1{x_1}+\cdots+\gamma_N{x_N}\in \K[x_{1},\dots,x_{N}].
\end{displaymath}
Set $\delta=[\K:\Q]$ and $h_{0}=\max_{j}\h(\gamma_{j})$. 

Let $\ag=(a_1,\ldots,a_N)\in\Z^N$ such that the univariate Laurent polynomial
$$
f=L(t^{a_1},\ldots,t^{a_N})=\gamma_0+\gamma_1t^{a_1}+\cdots+\gamma_Nt^{a_N}
$$ 
is nonzero and has a multiple root at a point $\xi\in \Qb\setminus \upmu_{\infty}$. Set $a_0=0$ and assume for the moment that
\begin{equation}
\label{subsums}
\xi \hbox{ is {\it not} a multiple root of }
\sum_{j\in\Lambda}\gamma_jt^{a_j} 
\hbox{ for every nonempty } \Lambda\subsetneq\{0,\ldots,N\}.
\end{equation}

We remark that $(a_1,\ldots,a_N)\neq(0,\ldots,0)$, since otherwise $f$ is a nonzero constant and 
cannot vanish at $\xi$. Set $d=\gcd(a_1,\ldots,a_N)$ and put $a'_j=a_j/d$, $j=1,\ldots,N$. We
apply Theorem~\ref{gcd-ALS+} to the polynomials
\begin{displaymath}
f_1=\gamma_0+\gamma_1t^{a'_1}+\cdots+\gamma_Nt^{a'_N} \quad \text{ and
} \quad 
f_2=tf'_1=\gamma_1a'_1t^{a'_1}+\cdots+\gamma_Na'_Nt^{a'_N},
\end{displaymath}
and the homomorphism $\varphi\colon
\Gm\to \Gm^{N}$ defined by $ \varphi(t)=(t^{a'_1},\ldots,t^{a'_N}) $.

Thus $f=f_1(t^d)$ and, in the notation of Theorem~\ref{gcd-ALS+}, $D=\vert\ag'\vert$, 
\begin{displaymath}
L_1=\gamma_0+\gamma_1x_1+\cdots+\gamma_Nx_N \quad
\text{ and } \quad 
L_2=\gamma_1a'_1x_1+\cdots+\gamma_Na'_Nx_N.
\end{displaymath}
We have
\begin{displaymath}
  \h(f_{i}) \le h_{0}+\log(D).
\end{displaymath}
Let $B'=B'(N,\delta)$ be the constant which appears in
Theorem~\ref{gcd-ALS+}. If the inequality~(\ref{main-inequa}) is not
satisfied, we have
$$
D^{\frac{1}{2(N-1)}}\leq B'\cdot (1+h_0+\log (D)),
$$
which shows that $D\leq C_1$ for some positive constant
$C_1=C_1(N,\delta,h_0)$. In this case, we choose $k=N-1$,
$\bg_j=a'_j$, $j=1,\ldots,N$, $\theta_1=d$ and $C\geq C_1$. Assertions
\eqref{item:1}, \eqref{item:2} and \eqref{item:3} of
Theorem~\ref{multiple} are then clearly verified.  

We now assume that the inequality~\eqref{main-inequa} is
satisfied. Theorem~\ref{gcd-ALS+} then gives a {nonnegative} integer
$k\le N$ and two morphisms $\psi\colon\Gm^{N-k}\to\Gm^N$ and
$\varphi_1\colon\Gm\to\Gm^{N-k}$ satisfying the conditions
\eqref{item:7}, \eqref{item:8} and \eqref{item:9} of that
theorem. Write $\psi(\yg)=(\yg^{\bg_1},\ldots,\yg^{\bg_N})$ and
$\varphi_1(t)=(t^{\theta'_1},\ldots,t^{\theta'_{N-k}})$ with
$\bg_1,\ldots,\bg_N\in\Z^{N-k}$ of size $\leq B'$ and
$\theta'_1,\ldots,\theta'_{N-k}\in\Z$ of size $\leq B' D$.  By
\eqref{item:7}, the $N\times (N-k)$ matrix $\Bg=(b_{j,i})$ {is
  primitive} and $\ag'=\Bg\cdot\thetag'$. We set
\begin{equation*}
{F_1}=\psi^\#(L_1)=\gamma_0+\gamma_1\yg^{\bg_1}+\cdots+\gamma_N\yg^{\bg_N},
\quad {F_2}=\psi^\#(L_2)=\gamma_1a'_1\yg^{\bg_1}+\cdots+\gamma_Na'_N\yg^{\bg_N},
\end{equation*}
and we consider the differential operator
$$
\Delta=\theta'_1y_1\frac{\partial}{\partial y_1}+\cdots+\theta'_{N-k}y_{N-k}\frac{\partial}{\partial y_{N-k}}.
$$
Let $\bg\in\Z^{N-k}$. The monomial $\yg^{\bg}$ is an
eigenvector of $\Delta$ with eigenvalue the scalar product
$\langle\bg , \thetag'\rangle$. Hence
$$
\Delta {F_1}=\sum_{i=1}^N\gamma_i \langle\bg_{i},\thetag'\rangle \yg^{\bg_i}={F_2}.
$$

Set $G=\gcd(F_{1},F_{2})$. By hypothesis,  $\xi^d$ is a common non-cyclotomic
root of $f_1$ and $f_2$ and, by the additional assumption~(\ref{subsums}), $\xi^d$ is not a multiple root
of $\sum_{j\in\Lambda}\gamma_jt^{a'_j}$ for any nonempty proper subset 
$\Lambda$ of $\{0,\ldots,N\}$. By
Theorem~\ref{gcd-ALS+}\eqref{item:9},  there exists an irreducible factor $P$ of $G$
such that $\pi=\varphi_1^\#(P) \in \K[t]$ vanishes at $\xi^d$.

We want to show that $P$ is a multiple factor of $G$. Since
$P\mid{F_1}$ and $P\mid F_{2}\Delta{F_1}$, by standard arguments
either $P^2\mid {F_1}$ as we want, or $\Delta P=\lambda P$ for a
constant $\lambda$. Let us assume  that this last
assertion holds. Write 
\begin{displaymath}
 P=\sum_{\bg\in \Z^{N-k}} c_{\bg}\yg^{\bg} 
\end{displaymath}
and set $\supp(P)=\{ \bg\in\Z^{N-k} \mid c_{\bg}\neq0\}$ for the
support of $P$.  
The
differential equation $\Delta P=\lambda P$ then says that the scalar
product $\langle\bg, \thetag'\rangle$ is constant over $\supp(P)$,
which in turns implies that $\pi$ is a monomial. 
But
then $\pi$ cannot vanish
 at
$\xi^{d}$ because the latter is nonzero, which is a contradiction. 

Thus $P$ is a multiple factor of $F_1$. Set $\theta_i=d\theta'_i$, so
that $P(t^{\theta_1},\ldots,t^{\theta_N})$ is a multiple factor of $f$
which vanishes at the point $\xi$, as required. {Remark that $k\geq 1$. Indeed 
the matrix $\Bg$ is primitive and the polynomial $L_1$ does not have multiple factors, since it 
is linear.} Theorem~\ref{multiple} thus follows, under the additional 
hypothesis~(\ref{subsums}), by choosing $C=\max\{C_1,B'\}$.\\

We now  explain how to remove the extra
assumption~(\ref{subsums}). Let as assume that~(\ref{subsums}) does
not hold.  We decompose $\{0,\ldots,N\}$ as a maximal union of
$u\geq2$ nonempty disjoint subsets $\Lambda_1,\ldots,\Lambda_u$ in
such a way that $\xi$ is a multiple root of
$\sum_{j\in\Lambda_i}\gamma_j t^{a_j}$ for $i=1,\ldots,u$. To simplify
the notation, we assume $u=2$ and $\Lambda_1=\{0,\ldots,M\}$ with
$0\leq M\le N-1$. Thus $\xi$ is a multiple root of both
\begin{equation}
\label{spezza}
\gamma_0+\sum_{j=1}^M\gamma_jt^{a_j}\quad{\rm and}\quad \sum_{j=M+1}^{N}\gamma_jt^{a_j}.
\end{equation}
Moreover, $\xi$ is {\sl not} a multiple root of
$\sum_{j\in\Delta}\gamma_jt^{a_j}$ for any nonempty $\Delta$ which is
a proper subset of $\{0,\ldots,M\}$ or of $\{M+1,\ldots,N\}$.

We write
\begin{multline*}
  \gamma_0+\gamma_1t^{a_1}+\cdots+\gamma_Nt^{a_N}
  =(\gamma_0+\gamma_1t^{a_1}+\cdots+\gamma_Mt^{a_M})\\
+t^{a_{M+1}}(\gamma_{M+1}+\gamma_{M+2}t^{a_{M+2}-a_{M+1}}+\cdots+\gamma_Nt^{a_N-a_{M+1}}).
\end{multline*}
We remark that $a_1,\ldots,a_M$, $a_{M+2}-a_{M+1},\ldots,a_N-a_{M+1}$ are not all zeros, since otherwise 
the polynomials~\eqref{spezza} are monomials vanishing at $\xi$, and hence they are both zero, which in turns 
implies that $f$ is also zero, contrary to the assumption of Theorem~\ref{multiple}.

Set $d=\gcd(a_1,\ldots,a_M,a_{M+2}-a_{M+1},\ldots,a_N-a_{M+1})$ and put 
$$
a'_j=
\begin{cases}
a_j/d, &\hbox{ for } j=1,\ldots,M,\\
(a_j-a_{M+1})/d, &\hbox{ for } j=M+3,\ldots,N.
\end{cases}
$$
Thus $a'_1,\ldots,a'_M,a'_{M+2},\ldots,a'_{N}$ are coprime, pairwise distinct, nonzero integers. We apply
Theorem~\ref{gcd-ALS+} to the homomorphism $\varphi\colon \Gm\to
\Gm^{N-1}$ defined by 
$$
\varphi(t)=(t^{a'_1},\ldots,t^{a'_M},t^{a'_{M+2}},\ldots,t^{a'_N})\;,
$$
and for the four polynomials
$$
f_1=\gamma_0+\sum_{j=1}^M\gamma_jt^{a'_j}, 
\quad f_2=\gamma_{M+1}+\sum_{j=M+2}^{N}\gamma_jt^{a'_j}, \quad
f_3=tf'_1, \quad  f_4=tf'_2.
$$
Thus $f=f_1(t^d)+t^{a_{M+1}}f_2(t^d)$ and $D=\vert\ag'\vert$. 

We argue as in the first part of the proof. We remark that 
$\h(f_{i}) \le h_{0}+\log(2D)$. Let $B'=B'(N,\delta)$ be the
constant that appears in Theorem~\ref{gcd-ALS+}.

If the inequality~(\ref{main-inequa}) is not satisfied, then $D\leq C_1=C_1(N,\delta,h_0)$. In this case, we choose $k=N-2$, $\theta_1=d$, $\theta_2=a_{M+1}$ and 
$$
\bg_j=
\begin{cases}
(a'_j,0) &\hbox{ for } j=1,\ldots,M,\\
(0,1) &\hbox{ for } j=M+1,\\
(a'_j,1) &\hbox{ for } j=M+2,\ldots,N.
\end{cases}
$$
Thus, in the notation of Theorem~\ref{multiple}\eqref{item:3},
$$
F=f_1(y_1)+y_2^{a_{M+1}}f_2(y_1)\in\Qb[y_{1}^{\pm1}, y_{2}^{\pm1}].
$$
Since $\xi$ is a multiple root of both $f_1(t^d)$ and $f_2(t^d)$, the
polynomials $f_1(y_1)$ and $f_2(y_1)$ have a common multiple factor,
say $P(y_1)$, which vanishes at $\xi^d$. Thus $P(y_1)$ is a multiple
factor of $F$ and $P(t^d)$ vanishes at $\xi$, as required.

It remains to consider the case when the
inequality~\eqref{main-inequa} is satisfied. Theorem~\ref{gcd-ALS+}
then gives a {nonnegative} integer $k\le N-1$, vectors
$\bg'_1,\ldots,\bg'_M,\bg'_{M+2},\ldots,\bg'_N\in\Z^{N-1-k}$ of size
$\leq B'$ and $\theta'_1,\ldots,\theta'_{N-1-k}\in\Z$ of size $\leq B'
D$ such that the $(N-1)\times (N-1-k)$ matrix $(b'_{j,i})_{j,i}$ has maximal
rank $N-1-k$ and $a'_j=\sum_{i=1}^{N-1-k}b'_{j,i}\theta'_i$ for
$j=1,\ldots,M$ and $j=M+2,\ldots,N$.  We set
$\yg=(y_1,\ldots,y_{N-1})$ and
\begin{align*}
F_1&=\gamma_0+\sum_{j=1}^M\gamma_j\yg^{\bg_j},\quad
&{F_2}&=\gamma_{M+1}+\sum_{j=M+2}^{N}\gamma_{j}\yg^{\bg_j},\\
{F_3}&=\sum_{j=1}^M\gamma_ja'_j\yg^{\bg_j},\quad
&{F_4}&=\sum_{j=M+2}^{N}\gamma_{j}a'_j\yg^{\bg_j},
\end{align*}
and  consider the differential operator
$$
\Delta=\theta'_1y_1\frac{\partial}{\partial y_1}+\cdots+\theta'_{N-1k}y_{N-1-k}\frac{\partial}{\partial y_{N-1-k}}.
$$
As in the first part of the proof, we have that $\Delta {F_1}=F_3$ and
$\Delta {F_2}=F_4$.

Set $G=\gcd(F_{1},F_{2},F_3,F_4)$ and write $f_i=\sum_{\alpha\in S}
f_{i,\alpha}t^{\alpha}$, $i=1,\ldots,4$, with
$$
S=\bigcup_{i=1}^4\supp(f_i)=\{0,a'_1,\ldots,a'_M,a'_{M+2},\ldots,a'_N\}.
$$ 
By hypothesis, $\xi^d$ is a common non-cyclotomic root of $f_1$,
$f_2$, $f_3$ and $f_4$.  We want to deduce from
Theorem~\ref{gcd-ALS+}\eqref{item:9} that $\varphi_1^\#(G) $ vanishes
at $\xi^d$. This certainly happens unless there exists a nonempty
proper subset $\Gamma$ of $S$ such that $\xi^d$ is a common root of
$\sum_{\alpha\in\Gamma}f_{i,\alpha}t^{\alpha}$, $i=1,\ldots,4$.

Assume by contradiction that this is the case. Then $\xi^d$ is a
multiple root of $\sum_{\alpha\in\Gamma}f_{i,\alpha}t^{\alpha}$,
$i=1$,~$2$.  We recall that
$$
\supp(f_1)=\{0,a'_1,\ldots,a'_M\},\qquad\supp(f_2)=\{0,a'_{M+2},\ldots,a'_{N}\}.
$$
Since $\xi$ is not a multiple root of
$\sum_{j\in\Delta}\gamma_jt^{a_j}$ for any nonempty $\Delta$ which is
a proper subset of $\{0,\ldots,M\}$ or of $\{M+1,\ldots,N\}$, we have
$$
\Gamma\cap\supp(f_1)=\emptyset\qquad{\rm or}\qquad\Gamma\cap\supp(f_1)=\supp(f_1)
$$
and
$$
\Gamma\cap\supp(f_2)=\emptyset\qquad{\rm or}\qquad\Gamma\cap\supp(f_2)=\supp(f_2).
$$
Since $\supp(f_1)\cap\supp(f_2)\neq\emptyset$, we deduce that
$\Gamma=\supp(f_1)\cup\supp(f_2)$, which contradict the previous
assumption.  Thus, by Theorem~\ref{gcd-ALS+}\eqref{item:9},
$\varphi_1^\#(G) $ vanishes at $\xi^d$.

Let $P$ be an irreducible factor of $G$ such that
$\pi=\varphi_1^\#(P) \in \K[t]$ vanishes at $\xi^d$. As in the first
part of the proof, $P$ is a multiple factor of both $F_1$ and $F_2$
and thus of the polynomial
$$
F=\gamma_0+\gamma_1\wt\yg^{\bg_1}+\cdots+\gamma_N\wt\yg^{\bg_N}=F_1(y_1,\ldots,y_{N-1})+y_N^{a_{M+1}}F_2(y_1,\ldots,y_{N-1})
$$
with $\wt \yg=(y_1,\ldots,y_N)$. Set $\theta_i=d\theta'_i$ for
$i=1,\ldots,N-1-k$, $\theta_{N-k}=a_{M+1}$ and
$$
\bg_j=
\begin{cases}
(b'_{j,1},\ldots,b'_{j,N-1-k},0) &\hbox{ for } j=1,\ldots,M,\\
(0,\ldots,0,1) &\hbox{ for } j=M+1,\\
(b'_{j,1},\ldots,b'_{j,N-1-k},1) &\hbox{ for } j=M+2,\ldots,N.
\end{cases}
$$
Then the $N\times (N-k)$ matrix $\Bg=(b_{j,i})_{j,i}$ has maximal rank and
$\ag=\Bg\cdot\thetag$, so that
$P(t^{\theta_1},\ldots,t^{\theta_{N-1}})$ is a multiple factor of $f$
which vanishes at the point $\xi$.  Theorem~\ref{multiple} then follows by
choosing $C=\max\{C_1,B'\}$.

\section{On a conjecture of Bolognesi and Pirola} \label{proof-pirola}

Let $\varphi\colon\Gm\rightarrow\Gm^N$ be a homomorphism given by
$\varphi(t)=(t^{a_1},\ldots,t^{a_N})$ for a sequence of integers
$a_1,\ldots,a_N$ such that  $0<a_1<\cdots<a_N$, and consider the curve
$\Cu=\Im(\varphi)$. It is easy to verify  that the linear
subspace $X\subset \C^{N-1}$ defined by the condition
$$
{\rank}\begin{pmatrix}
a_1&a_1^2&\cdots&a_1^{N-2}&x_1-1\\[0.1cm]
a_2&a_2^2&\cdots&a_2^{N-2}&x_2-1\\[0.1cm]
\vdots&\vdots&\cdots&\vdots&\vdots\\[0.1cm]
a_N&a_N^2&\cdots&a_N^{N-2}&x_N-1
\end{pmatrix}<N-1
$$
has codimension $2$, and that the restriction of its defining equations to
$(t^{a_1},\ldots,t^{a_N})$ vanish to order $N-1$ at $t=1$.  Thus, $X$
is the osculating $(N-2)$-linear dimensional space of $\Cu$ at the point
$(1,\ldots,1)\in\Gm^N$.

It is convenient to homogenize by letting $a_0=0$ and considering the
$(N+1)\times N$ matrix given by 
$$
A(\ag,(x_0:\dots:x_N))
=\begin{pmatrix}
1&a_0&a_0^2&\cdots&a_0^{N-2}&x_0\\[0.1cm]
1&a_1&a_1^2&\cdots&a_1^{N-2}&x_1\\[0.1cm]
\vdots&\vdots&\vdots&\cdots&\vdots&\vdots\\[0.1cm]
1&a_N&a_N^2&\cdots&a_N^{N-2}&x_N
\end{pmatrix}.
$$
Then we identify $X$ with the linear subspace of $\P^N$ defined by
condition
\begin{displaymath}
 \rank (A(\ag,(x_0:\dots:x_N)))<N. 
\end{displaymath}

For simplicity, we assume that $a_1,\ldots,a_N$ are coprime. Then $L$
intersects $\Cu$ in a second point different from the osculating one
if and only if there exists $\xi\neq1$ such that
$\rank(A(\ag,(1:\xi^{a_1}:\ldots:\xi^{a_N})))<N$. In~\cite{BolognesiPirola:osde},
Bolognesi and Pirola conjecture that this can never happen. It easily
seen that, to prove their conjecture, we may assume that $\xi$ is not
torsion.

In the case $N=2$ the conjecture is trivial. Bolognesi and Pirola
proved the conjecture for $N=3$. In~\cite{CorvajaZannier:rm}, Corvaja
and Zannier proved a weak form of the conjecture for $N=4$, namely
that the set of exceptional pairs $(\ag,\xi)$ such that the matrix
$A(\ag,(1:\xi^{a_1}:\ldots:\xi^{a_N}))$ has rank $<N$ is finite.

As a second application of Theorem~\ref{gcd-ALS+}, we prove the
following result.

\begin{theorem}
\label{pirola}
There is a constant $C$ depending only on $N$ such that the following
holds. 

Let $a_1,\ldots,a_N$ be integers such that $0=a_0<a_1 < a_2 < \cdots <
a_N=:D$ and $\xi\in \Qb^{\times}\setminus \upmu_{\infty}$.  If the matrix
 $A(\ag,(1:\xi^{a_1}:\ldots:\xi^{a_N}))$ has rank
$<N$, then there exist $1 \le k\le N-1$ and vectors
$\bg_1,\ldots,\bg_N,\thetag\in\Z^{N-k}$ such that
\begin{enumerate}
\item \label{item:14} $\vert\bg_{i}\vert\leq C$, $i=1,\dots, N$, and
  $\vert\thetag\vert\le C D$;
\item \label{item:15} the matrix
  $\Bg=(b_{i,j})_{i,j}\in \Z^{N\times(N-k)}$ {is primitive} and $\ag=\Bg\cdot\thetag$;
\item \label{item:16} the subvariety of $\Gm^{N-k}$ defined by 
$$
V=\{\yg\in\Gm^{N-k} \mid \rank (A(\ag,(1:\yg^{\bg_1}:\ldots:\yg^{\bg_N})))<N\}
$$
has a component of codimension $1$ containing the point
$(\xi^{\theta_1},\ldots,\xi^{\theta_{N-k}})$.
\end{enumerate}
\end{theorem}

\begin{proof}
  The proof is very similar to that of Theorem~\ref{multiple}.  
  
Let $0=a_0<a_1 < a_2 < \cdots < a_N=:D$ and $\xi\in \Qb^{\times}\setminus
\upmu_{\infty}$ such that the matrix $A(\ag,(1:\xi^{a_1}:\ldots:\xi^{a_N}))$ has rank $<N$.  
For each subset $\Lambda\subset\{0,\ldots,N\}$, we put $v_{\Lambda,j}=
\xi^{a_j}$ if $ j\in\Lambda$ and $v_{\Lambda,j}=0$ otherwise. Then we assume that
\begin{multline}
\label{subsums2}
\hbox{for all nonempty }\Lambda\subsetneq\{0,\ldots,N\}, 
\rank(A(\ag,(v_{\Lambda,0}:v_{\Lambda,1}:\ldots:v_{\Lambda,N})))=N.
\end{multline}
This extra assumption may be removed, proceeding as in the last part
of the proof of Theorem~\ref{multiple}.

Let $d=\gcd(a_1,\ldots,a_N)$ and put $a'_i=a_i/d$, $i=1,\ldots,N$. As
in the proof of Theorem~\ref{multiple}, we may assume, by replacing
$\ag$ by $\ag'$, that $d=1$.

As already remarked, the linear space $X$ defined by 
$$
\rank(A(\ag,(x_0:x_1:\ldots:x_N)))<N
$$
is defined by two linear equations, say 
$$
L_i=\gamma_{i,0}x_0+\gamma_{i,1}x_1+\cdots+\gamma_{i,N}x_n,\qquad i=1,2,
$$ 
with coefficients $\gamma_{i,j}$ bounded by $N!D^{N^2}$. We apply Theorem~\ref{gcd-ALS+}, 
choosing $K_0=\Q$, $s=2$ and $\varphi(t)=(t^{a_1},\ldots,t^{a_N})$.  Thus 
$$
f_i=\gamma_{i,0}+\gamma_{i,1}t^{a_1}+\cdots+\gamma_{i,N}t^{a_N},\qquad i=1,2.
$$ 
These two polynomials are not both zeros, since otherwise 
$$
\rank(A(\ag,(1:t^{a_1}:\ldots:t^{a_N})))<N
$$ 
identically, which is not possible by the assumption $0<a_1 < a_2 < \cdots < a_N$.

Let $B'=B'(N,1)$ be the constant which appears in Theorem~\ref{gcd-ALS+}. If the inequality~(\ref{main-inequa}) of that 
theorem is not satisfied, we have that
$$
D^{\frac{1}{2(N-1)}}\leq B'\cdot (1+N^2\log D+N\log N),
$$
which shows that $D\leq C'_1$ for some positive constant
$C'_1=C'_1(N)$. In this case we simply choose $k=N-1$, $\bg_i=a_i$ for
$i=1,\ldots,N$ and $\theta_1=1$. Assertions \eqref{item:14},
\eqref{item:15} and~\eqref{item:16} of Theorem~\ref{pirola} are
clearly verified for $C'\geq C'_1$.

Thus we may assume that the inequality~(\ref{main-inequa}) is
satisfied. Theorem~\ref{gcd-ALS+} then gives a {nonnegative} integer $k<N$
and two homomorphisms $\psi\colon\Gm^{N-k}\to\Gm^N$ and
$\varphi_1\colon\Gm\to\Gm^{N-k}$ satisfying \eqref{item:7},
\eqref{item:8} and \eqref{item:9} of that theorem.  Let
$\psi(\yg)=(\yg^{\bg_1},\ldots,\yg^{\bg_N})$ and
$\varphi_1(t)=(t^{\theta_1},\ldots,t^{\theta_{N-k}})$ with
$\bg_1,\ldots,\bg_N\in\Z^{N-k}$ of size $\leq B'$ and
$\theta_1,\ldots,\theta_{N-k}\in\Z$ of size $\leq B' D$. By
Theorem~\ref{gcd-ALS+}\eqref{item:7}, the matrix $\Bg=(b_{i,j})_{i,j}$
is primitive and $\ag=\Bg\cdot\thetag$. 

By the assumption~(\ref{subsums2}), $\xi$ is not a common root of
$\sum_{j\in\Lambda}\gamma_{i,j}t^{a_j}$, $i=1,2$, for any nonempty
$\Lambda\subsetneq\{0,\ldots,N\}$. Thus, by
Theorem~\ref{gcd-ALS+}\eqref{item:9}, the greatest common divisor of
$F_1(\yg^{\bg_1},\ldots,\yg^{\bg_N})$ and
$F_2(\yg^{\bg_1},\ldots,\yg^{\bg_N})$ must vanish at
$(\xi^{\theta_1},\ldots,\xi^{\theta_{N-k}})$. This means that $V$ has
a component of codimension $1$ through the point
$(\xi^{\theta_1},\ldots,\xi^{\theta_{N-k}})$, as
required. {Since $X$ is a linear space of codimension $2$ and 
$\Bg$ is primitive, we must have $k\geq 1$.}
Theorem~\ref{pirola} follows by choosing
$C'=\max\{C'_1,B'\}$.
\end{proof}

\begin{remark} \label{rem:2} An immediate consequence of
  Theorem~\ref{pirola}(\ref{item:14},\ref{item:15}) is that the
  vectors $\ag$ such that the matrix
  $A(\ag,(1:\xi^{a_1}:\ldots:\xi^{a_N}))$ has rank $<N$ for some
  $\xi\in \Qb^{\times}\setminus \upmu_{\infty}$, lie on a finite union
  of proper vector subspaces of $\Q^N$, which is effectively
  computable for every given $N$. 

  Moreover, the condition \eqref{item:16} can be translated in terms
  of resultants, and can be checked by the search of integral points
  $\thetag=(\theta_1,\ldots,\theta_{N-k})\in\Z^{N-k}$ on a finite
  family of varieties, depending only on $N$.  More precisely, fix
  $k\in\{1,\ldots,N-1\}$ and fix one of the finitely many
  $N\times(N-k)$ {primitive matrix $\Bg=(b_{i,j})_{i,j}$ with}
  entries of size bounded by $C(N)$.  Let
  $F_i(\theta_1,\ldots,\theta_{N-k};y_1,\ldots,y_{N-k})$, $i=1,2$, be
  any two distinct $(N-1)\times(N-1)$ determinants of the matrix
$$
\begin{pmatrix}
\sum b_{1j}\theta_j&(\sum b_{1j}\theta_j)^2&\cdots&(\sum b_{1j}\theta_j)^{N-2}&\yg^{b_1}-1\\[0.1cm]
\sum b_{2j}\theta_j&(\sum b_{2j}\theta_j)^2&\cdots&(\sum b_{2j}\theta_j)^{N-2}&\yg^{b_2}-1\\[0.1cm]
\vdots&\vdots&\cdots&\vdots&\vdots\\[0.1cm]
\sum b_{Nj}\theta_j&(\sum b_{Nj}\theta_j)^2&\cdots&(\sum b_{Nj}\theta_j)^{N-2}&\yg^{b_N}-1
\end{pmatrix}.
$$
Compute the resultant $R\in\Z[\thetag][y_1,\ldots,y_{N-k-1}]$ of $F_1$ and $F_2$ with respect to, say, $y_{N-k}$ and let $W$ be the variety defined by the vanishing of the coefficients of $R$, viewed as a polynomial in the variables $y_1,\ldots,y_{N-k-1}$. Then $V$ has a component of codimension $1$ if and only if $\thetag\in W$ and $\ag=\Bg\cdot\thetag$.
\end{remark}

\bibliographystyle{amsalpha}
\bibliography{biblio}

\end{document}